\documentclass{amsart}

\usepackage[colorlinks=true, urlcolor=blue, citecolor=blue, linkcolor=blue, hyperfootnotes=true]{hyperref}
\usepackage{amssymb}
\usepackage{aliascnt}

\newcommand{\norm}[1]{\left\lVert #1 \right\rVert}

\newtheorem{thm}{Theorem}[section]

\newaliascnt{prp}{thm}
\newtheorem{prp}[prp]{Proposition}
\aliascntresetthe{prp}

\newaliascnt{cor}{thm}
\newtheorem{cor}[cor]{Corollary}
\aliascntresetthe{cor}

\newaliascnt{lem}{thm}
\newtheorem{lem}[lem]{Lemma}
\aliascntresetthe{lem}

\theoremstyle{definition}

\newaliascnt{dfn}{thm}
\newtheorem{dfn}[dfn]{Definition}
\aliascntresetthe{dfn}

\newtheorem{qst}{Question}

\numberwithin{equation}{section}

\newcounter{cnt}

\author{Tristan Bice}
\address{Institute of Mathematics of the Polish Academy of Sciences, Warsaw, Poland}
\email{tristan.bice@gmail.com}
\urladdr{http://www.tristanbice.com}
\author{Alessandro Vignati}
\address{Department of Mathematics and Statistics, York University, Toronto, Ontario, Canada}
\email{ale.vignati@gmail.com}
\urladdr{http://www.automorph.net/avignati}
\thanks{The first author has been supported by a WCMCS grant at IMPAN (Poland). The second author is partially supported by a York University Susan Mann Scholarship.}

\title{$\mathrm{C}^*$-Algebra Distance Filters}

\begin{document}

\begin{abstract}
We use non-symmetric distances to give a self-contained account of C*-algebra filters and their corresponding compact projections, simultaneously simplifying and extending their general theory.
\end{abstract}

\maketitle

\section*{Introduction}

Quantum filters were introduced by Farah and Weaver to analyze pure states on $\mathrm{C}^*$-algebras and various conjectures concerning them, like Anderson's conjecture and the Kadison-Singer conjecture (which has since become the Marcus-Speilman-Srivastava theorem \textendash\, see \cite{MarcusSpielmanSrivastava2013}).  They were also considered more recently in \cite{BlecherWeaver2016} in relation to quantum analogs of certain large cardinals, and they even make an appearance much earlier in \cite{AkemannPedersen1992} as faces of the positive unit ball.  While their basic theory was fleshed out in \cite{Bice2011} (as `norm filters') and \cite{FarahWofsey2012} (and in a forthcoming book by Farah), there remained some fundamental questions which we aim to address in this paper.

The first such question is why they should be considered as filters at all.  Filters in the classical sense are defined from a transitive relation (as the downwards directed upwards closed subsets), but in general there is no such relation defining quantum filters.  Indeed, it can even happen that every maximal quantum filter in a $\mathrm{C}^*$-algebra fails to be a filter in the traditional order theoretic sense \textendash\, see \cite[Corollary 6.6]{Bice2011}.  While it might be intuitively clear that quantum filters are the `right' quantum analog, and their utility in analyzing states justifies their study, regardless of whether they are considered as filters or not, a more precise connection to order theory would of course be desirable.

The key here is to replace classical transitive relations with `continuous' ones.  These are the non-symmetric distances, binary functions $\mathbf{D}$ to $[0,\infty]$ satisfying the continuous version of transitivity, namely the triangle inequality
\[\mathbf{D}(x,y)\leq\mathbf{D}(x,z)+\mathbf{D}(z,y).\]
The first order sentences defining classical filters also have continuous versions, with the quantifiers $\forall$ and $\exists$ replaced by suprema and infima respectively.  Then quantum filters are indeed the continuous filters with respect to the appropriate distance $\mathbf{d}$ on the positive unit ball $A^1_+$, namely
\[\mathbf{d}(a,b)=\norm{a-ab}.\]
This simple observation allows for a markedly different approach to the theory.

In \autoref{D} we start off by examining the relationship between various distances and distance-like functions.  As with metrics, uniform equivalence plays a fundamental role.  We move on to $\mathbf{d}$-filters in \autoref{F}, using these relationships to provide characterizations using the distance, order, multiplicative and convex structure of $A^1_+$.  In \autoref{CP} we then show how $\mathbf{d}$-filters in $A$ represent compact projections in $A^{**}$ (just as hereditary $\mathrm{C}^*$-subalgebras in $A$ represent open projections in $A^{**}$).  We finish by examining interior containment of compact projections and its relation to the reverse Hausdorff distance on $\mathbf{d}$-filters.

\section{Distances}\label{D}

We will deal with a number of binary functions $\mathbf{D}$ from some set $X$ to $[0,\infty]$.  We view these as `generalized' or `continuous' relations on $X$.  More precisely, any $\mathbf{D}:X\times X\rightarrow[0,\infty]$ defines a classical relation $|\mathbf{D}|\subseteq X\times X$ by
\[x|\mathbf{D}|y\quad\Leftrightarrow\quad\mathbf{D}(x,y)=0.\footnote{Note we are using the standard infix notation $xRy$ to mean $(x,y)\in R$.}\]
We say that the function $\mathbf{D}$ \emph{quantifies} the relation $|\mathbf{D}|$.  Conversely, every relation $R\subseteq X\times X$ has a trivial quantification given by its characteristic function, which we also denote by $R$, specifically
\[
R(x,y)=\begin{cases}0&\text{if }xRy\\ \infty&\text{otherwise}.\end{cases}
\]
We define the \emph{composition} $\mathbf{D}\circ\mathbf{E}$ of any $\mathbf{D},\mathbf{E}:X\times X\rightarrow[0,\infty]$,  by
\[
(\mathbf{D}\circ\mathbf{E})(x,y)=\inf_{z\in X}(\mathbf{D}(x,z)+\mathbf{E}(z,y)).
\]
Note when $R$ and $S$ are relations (identified with their characteristic functions),
\[
x(R\circ S)y\quad\Leftrightarrow\quad\exists z\in X\ (xRzSy),
\]
so $\circ$ extends the usual composition of classical relations.\footnote{In other words, the category $\mathbf{Rel}$ of classical relations forms a wide subcategory of $\mathbf{GRel}$, the category of generalized relations \textendash\, see \cite[\S1]{BDD} for more details.} Moreover, we always have
\[|\mathbf{D}|\circ|\mathbf{E}|\subseteq|\mathbf{D}\circ\mathbf{E}|.\]
We say that $\mathbf{D}$ is \emph{$\mathbf{E}$-invariant} when $\mathbf{D}=\mathbf{D}\circ\mathbf{E}=\mathbf{E}\circ\mathbf{D}$.

\begin{dfn}
$\mathbf{D}$ is a \emph{distance} if
\[\label{tri}\tag{$\triangle$}\mathbf{D}\leq\mathbf{D}\circ\mathbf{D}.\]
\end{dfn}

On a $\mathrm{C}^*$-algebra $A$, the only distance usually considered is the metric given by
\[\mathbf{e}(x,y)=\norm{x-y}.\]
Indeed, metrics are precisely the symmetric distances quantifying the equality relation.  However, our thesis is that one should also consider various non-symmetric distances on $\mathrm{C}^*$-algebras which quantify other important order relations like
\begin{align*}
a\ll b\qquad&\Leftrightarrow\qquad a=ab.\\
a\leq b\qquad&\Leftrightarrow\qquad b-a\in A_+.
\end{align*}
Here $A_+$ denotes the positive elements in $A$, while $A_\mathrm{sa}$, $A^r$ and $A^{=r}$ will denote the self-adjoints, $r$-ball and $r$-sphere respectively.  We also consider $A$ embedded canonically in its enveloping von Neumann algebra $A^{**}$ and set
\[\widetilde{A}=A+\mathbb{C}1.\]
So if $A$ is unital then $\widetilde A=A$, otherwise $\widetilde{A}$ is the unitization of $A$ (see \cite[II.1.2]{Blackadar2017}).

\begin{prp}
\begin{align}
\label{ddist}\mathbf{d}(a,b)&=\norm{a-ab}\quad\text{is an $\mathbf{e}$-invariant distance on }A^1_+\text{ quantifying }\ll.\\
\label{hdist}\mathbf{h}(a,b)&=\norm{(a-b)_+}\quad\text{is an $\mathbf{e}$-invariant distance on }A_\mathrm{sa}\text{ quantifying }\leq.
\end{align}
\end{prp}

\begin{proof}\
\begin{itemize}
\item[\eqref{ddist}] For $a,b,c\in A^1_+$, we have $\norm{a},\norm{b^\perp}\leq1$, where $b^\perp=1-b(\in\widetilde{A}$), so
\[\mathbf{d}(a,b)=\norm{ab^\perp}=\norm{a(c^\perp+c)b^\perp}\leq\norm{ac^\perp}\norm{b^\perp}+\norm{a}\norm{cb^\perp}\leq\mathbf{d}(a,c)+\mathbf{d}(c,b).\]
As $\mathbf{e}$ quantifies equality, we immediately have $\mathbf{d}\circ\mathbf{e},\mathbf{e}\circ\mathbf{d}\leq\mathbf{d}$.  Conversely,
\begin{align*}
\mathbf{d}(a,b)&=\norm{a-ab}\leq\norm{a-c}+\norm{c-cb}=\mathbf{e}(a,c)+\mathbf{d}(c,b).\\
\mathbf{d}(a,b)&=\norm{a-ab}\leq\norm{a-ac}+\norm{ac-ab}\leq\mathbf{d}(a,c)+\mathbf{e}(c,b).
\end{align*}

\item[\eqref{hdist}]  Denote the space of quasistates on $A$ by $\mathsf{Q}=(A^{*1}_+=$ positive linear functionals in the dual unit ball$)$ and recall that, for $a\in A_\mathrm{sa}$,
\begin{equation}\label{||a+||}
\norm{a_+}=\sup_{\phi\in\mathsf{Q}}\phi(a).
\end{equation}
Thus for all $a,b,c\in A_{sa}$, we have $\mathbf{h}(a,b)\leq\sup\limits_{\phi\in\mathsf{Q}}\phi(a-c)+\sup\limits_{\phi\in\mathsf{Q}}\phi(c-b)=\mathbf{h}(a,c)+\mathbf{h}(c,b)$.\footnote{One might naively use $(a+b)_+\leq a_++b_+$ instead, but this only holds for commutative $A$.}  Now as $\norm{a_+}\leq\norm{a}$, we have $\mathbf{h}\leq\mathbf{e}$ so $\mathbf{h}\leq\mathbf{h}\circ\mathbf{h}\leq\mathbf{h}\circ\mathbf{e},\mathbf{e}\circ\mathbf{h}$.  But the reverse inequalities are  again immediate, as $\mathbf{e}$ quantifies equality.\qedhere
\end{itemize}
\end{proof}

Basic relationships between $\mathrm{C}^*$-algebra distances reveal aspects of $\mathrm{C}^*$-algebraic structure.  Here are some required for our investigation of $\mathrm{C}^*$-algebra filters.

\begin{prp}
On $A^1_+$,
\begin{align}
\label{h<2d}\mathbf{h}\ &\leq\ 2\mathbf{d}.\\
\label{d2<hd}\mathbf{d}^2\ &\leq\ \mathbf{d}\circ\mathbf{h}.\\
\label{d2<dh}\mathbf{d}^2\ &\leq\ \mathbf{h}\circ\mathbf{d}.
\end{align}
In fact, in \eqref{d2<hd} and \eqref{d2<dh}, we can even take $\circ$ in $A_\mathrm{sa}$.
\end{prp}

\begin{proof}\
\begin{itemize}
\item[\eqref{h<2d}]  For $a,b\in A^1_+$, $bab\leq b^2\leq b$ so
\begin{align*}
\mathbf{h}(a,b)&\leq\mathbf{h}(a,bab)+\mathbf{h}(bab,b)\\
&=\norm{a-bab}\\
&\leq\norm{a-ab}+\norm{ab-bab}\\
&\leq\norm{a-ab}+\norm{a-ba}\norm{b}\\
&\leq2\mathbf{d}(a,b).
\end{align*}

\item[\eqref{d2<hd}]  First note that, for any $a\in A^1$ and $b\in A_\mathrm{sa}$,
\begin{equation}\label{aba+}
\norm{(aba)_+}=\inf_{aba\leq c}\norm{c}\leq\inf_{b\leq c}\norm{aca}\leq\norm{ab_+a}\leq\norm{b_+}.
\end{equation}
As $\norm{(a+b)_+}\leq\norm{a_+}+\norm{b_+}$ (see \eqref{||a+||}), for all $a,b\in A^1_+$ and $c\in A_\mathrm{sa}$,
\begin{align*}
\mathbf{d}(a,b)^2&=\norm{ab^{\perp2}a}\\
&\leq\norm{ab^\perp a}=\norm{(ab^\perp a)_+}\\
&\leq\norm{(ac^\perp a)_+}+\norm{(a(b^\perp-c^\perp)a)_+}\\
&\leq\norm{ac^\perp}\norm{a}+\norm{(c-b)_+}\\
&\leq\mathbf{d}(a,c)+\mathbf{h}(c,b).
\end{align*}

\item[\eqref{d2<dh}]  Likewise, for $a,b\in A^1_+$ and $c\in A_\mathrm{sa}$,
\begin{align*}
\mathbf{d}(a,b)^2&=\norm{b^\perp a^2b^\perp}\\
&\leq\norm{b^\perp ab^\perp}=\norm{(b^\perp ab^\perp)_+}\\
&\leq\norm{(b^\perp(a-c)b^\perp)_+}+\norm{(b^\perp cb^\perp)_+}\\
&\leq\norm{(a-c)_+}+\norm{b^\perp}\norm{cb^\perp}\\
&\leq\mathbf{h}(a,c)+\mathbf{d}(c,b).\qedhere
\end{align*}
\end{itemize}
\end{proof}

Another important quantification of $\leq$ on $A_+$ is given by
\[\mathbf{p}(a,b)=\inf_{0\leq c\leq b}\norm{a-c}.\]
Often $\mathbf{p}(a,b)=\mathbf{h}(a,b)$, e.g. if $ab=ba$, $a\leq b$, $b\leq a$ or if $a$ and $b$ are projections.  However, $\mathbf{p}$ and $\mathbf{h}$ do not coincide in general.

\begin{prp}
On $M_{2+}$, $\mathbf{p}\neq\mathbf{h}$.  In fact, $\mathbf{p}$ is not even a distance.
\end{prp}

\begin{proof}
Let $a=\begin{bmatrix}1&1\\1&1\end{bmatrix}$ and $b=\begin{bmatrix}4&0\\0&0\end{bmatrix}$ so $a-tb=\begin{bmatrix}1-4t&1\\1&1\end{bmatrix}$ and
\[\mathrm{det}(a-tb-\lambda)=(1-4t-\lambda)(1-\lambda)-1=\lambda^2+(4t-2)\lambda-4t.\]
So $(a-tb)$ has eigenvalues $1-2t\pm\sqrt{4t^2+1})$ and hence
\[\norm{a-tb}=\begin{cases}1-2t+\sqrt{4t^2+1} &\text{for }t\leq\frac{1}{2}\\ 1-2t-\sqrt{4t^2+1} &\text{for }t\geq\frac{1}{2}\end{cases}.\]
Thus $\mathbf{p}(a,b)=\inf_{t\in[0,1]}\norm{a-tb}=\norm{a-\frac{1}{2}b}=\sqrt{2}$.  On the other hand, $\mathbf{h}(a,b)$ is the positive eigenvalue for $t=1$, i.e., $\sqrt{5}-1$ so
\[\mathbf{h}(a,b)<\mathbf{p}(a,b).\]
Let $c=a+(b-a)_+$ so $\mathbf{p}(a,c)=0$ and $\mathbf{p}(c,b)=\norm{c-b}=\norm{(a-b)_+}=\sqrt{5}-1$, as $a,b\leq c$.  Thus
\[\mathbf{p}(a,b)>\mathbf{p}(a,c)+\mathbf{p}(c,b).\qedhere\]
\end{proof}

Even when two metrics differ, it often suffices to show they are uniformly equivalent.  If $\mathbf F$ and $\mathbf G$ are functions $\mathbf{F},\mathbf{G}:X\rightarrow[0,\infty]$ we define
\[\mathbf{F}\precapprox\mathbf{G}\qquad\Leftrightarrow\qquad0=\lim_{r\rightarrow0}\sup_{\mathbf{G}(x)\leq r}\mathbf{F}(x).\]
Equivalently, $\mathbf{F}\precapprox\mathbf{G}$ if and only if, for all $Y\subseteq X$,
\[\inf_{y\in Y}\mathbf{G}(y)=0\quad\Rightarrow\quad\inf_{y\in Y}\mathbf{F}(y)=0.\]
We call $\mathbf{F}$ and $\mathbf{G}$ \emph{uniformly equivalent}, written $\mathbf{F}\approx\mathbf{G}$, when $\mathbf{F}\precapprox\mathbf{G}\precapprox\mathbf{F}$.

\begin{thm}\label{h=p}
On $A_+^1\times A_+^1$, $\mathbf{h}\leq\mathbf{p}\precapprox\mathbf{h}$, so $\mathbf{h}$ is uniformly equivalent to $\mathbf{p}$. 
\end{thm}

\begin{proof}
First we show that $\mathbf{h}\leq\mathbf{p}$ on $A^1_+$.  For any $c\leq b$, \eqref{||a+||} yields
\[\mathbf{h}(a,b)=\sup_{\phi\in\mathsf{Q}}\phi(a-b)\leq\sup_{\phi\in\mathsf{Q}}\phi(a-c)\leq\norm{a-c}.\]
Also $a-b\leq(a-b)_+$ so $a-(a-b)_+\leq b$ and $\mathbf{h}(a,b)=\norm{a-(a-(a-b)_+)}$ so
\[\mathbf{h}(a,b)=\inf_{c\leq b}\norm{a-c}\leq\inf_{0\leq c\leq b}\norm{a-c}=\mathbf{p}(a,b).\]

Next we need to show that
\[\forall\epsilon>0\quad\exists\delta>0\quad\forall a,b\in A^1_+\quad\mathbf{h}(a,b)<\delta\ \Rightarrow\ \mathbf{p}(a,b)<\epsilon.\]
Take $a,b\in A^1_+$ and let $z=b+(a-b)_+$ so $a,b\leq z\in A_+$ and
\[\sqrt{a}(\tfrac{1}{n}+z)^{-\frac{1}{2}}\sqrt{b}\rightarrow u,\]
for some $u\in A$, by \cite[Lemma 1.4.4]{Pedersen1979}.  As $\sqrt{a}(\frac{1}{n}+z)^{-\frac{1}{2}}\sqrt{z}\rightarrow\sqrt{a}$, we have
\[\sqrt{a}(\tfrac{1}{n}+z)^{-\frac{1}{2}}(\sqrt{b}-\sqrt{z})\rightarrow u-\sqrt{a}.\]
By the continuity of the continuous functional calculus, we have some function $O$ on $[0,2]$ with $\lim_{t\rightarrow0}O(t)=0$ such that $\norm{\sqrt{x}-\sqrt{y}}\leq O(\norm{x-y})$, for all $x,y\in A^1_+$.  For all $n$, we have $\norm{\sqrt{a}(\tfrac{1}{n}+z)^{-\frac{1}{2}}}\leq\norm{\sqrt{a}(\tfrac{1}{n}+a)^{-\frac{1}{2}}}\leq1$ so
\[\norm{u-\sqrt{a}}\leq\norm{\sqrt{b}-\sqrt{z}}\leq O(\norm{b-z})=O(\norm{(a-b)_+})=O(\mathbf{h}(a,b)).\]
As in the proof of \cite{Pedersen1979} Proposition 1.4.10, we have $u^*u\leq b$, and
\[\norm{a-u^*u}\leq\norm{a-u^*\sqrt{a}+u^*\sqrt{a}-u^*u}\leq\norm{\sqrt{a}-u^*}+\norm{\sqrt{a}-u}\leq2O(\mathbf{h}(a,b)).\]
Thus $\mathbf{p}\precapprox\mathbf{h}$.
\end{proof}

Another quantification of $\leq$ on $A^1_+$ is given by
\[\mathbf{n}(a,b)=\inf_{a\leq c\leq 1}\norm{c-b}.\]
For unital $A$, $\mathbf{n}(a,b)=\mathbf{p}(b^\perp,a^\perp)$ and $\mathbf{h}(a,b)=\mathbf{h}(b^\perp,a^\perp)$ so it immediately follows that $\mathbf{h}$ and $\mathbf{n}$ are also uniformly equivalent.  Actually, this extends to non-unital $A$.

\begin{cor}\label{h=n}
On $A^1_+\times A^1_+$, $\mathbf{h}\leq\mathbf{n}\precapprox\mathbf{h}$, so $\mathbf{h}$ is uniformly equivalent to $\mathbf{n}$.
\end{cor}

\begin{proof}
We can essentially apply the proof of \autoref{h=p} with $a$ and $b$ replaced by $b^\perp$ and $a^\perp$ respectively.  First, for any $c\geq a$, \eqref{||a+||} yields
\[\mathbf{h}(a,b)=\sup_{\phi\in\mathsf{Q}}\phi(a-b)\leq\sup_{\phi\in\mathsf{Q}}\phi(c-b)\leq\norm{c-b}.\]
Also $a-b\leq(a-b)_+$ so $a\leq b+(a-b)_+$ and $\mathbf{h}(a,b)=\norm{b+(a-b)_+-b}$ so
\[\mathbf{h}(a,b)=\inf_{a\leq c}\norm{c-b}\leq\inf_{0\leq c\leq b}\norm{c-b}=\mathbf{n}(a,b).\]
Now let $z=b^\perp+(a-b)_+$ and $u=\lim\sqrt{a^\perp}(\tfrac{1}{n}+z)^{-\frac{1}{2}}\sqrt{b^\perp}$.  As before, we have $a\leq(u^*u)^\perp$ and $\norm{(u^*u)^\perp-b}\leq2O(\mathbf{h}(a,b))$.  Note that $(u^*u)^\perp\in A$ even when $A$ is not unital, as then $\pi(z)=1=\pi(u)$ and hence $\pi((u^*u)^\perp)=0$, where $\pi$ is the character on the unitization $\widetilde{A}$ with kernel $A$.
\end{proof}

Restricting to projections $\mathcal{P}=\{p\in A:p\ll p^*\}$ allows for a stronger result.  First we need the following lemma, adapted from \cite{10152013}, which strengthens the standard result that close projections are unitarily equivalent (see \cite{Blackadar2017} II.3.3.5).

\begin{lem}\label{symex}
If $p,q\in\mathcal{P}$ and $\norm{p-q}<1$ then $p$ and $q$ can be exchanged by a symmetry(=self-adjoint unitary), i.e. we have $u\in\widetilde{A}_\mathrm{sa}$ with $u^2=1$ and $up=qu$.
\end{lem}

\begin{proof}
Let $a=p+q-1\in\widetilde{A}_\mathrm{sa}$ so $ap=qp=qa$ and $aq=pq=pa$.  Thus $a^2p=aqp=pqp=pqa=pa^2$.  Also, $a^2=pq+qp-p-q+1=1-(p-q)^2$ so $\norm{1-a^2}=\norm{p-q}^2<1$ and hence $a^2$ is invertible.  Thus we may set $u=a|a|^{-1}$ so $u^2=1$, as $a\in A_\mathrm{sa}$.  Also, as $p$ commutes with $a^2$ and hence with $|a|^{-1}$,
\[up=a|a|^{-1}p=ap|a|^{-1}=qa|a|^{-1}=qu.\qedhere\]
\end{proof}

\begin{cor}\label{epqhpq}
If $p,q\in\mathcal{P}$ and $\norm{p-q}<1$ then $\mathbf{h}(p,q)=\norm{p-q}$.
\end{cor}

\begin{proof}
if $\mathbf{e}(p,q)<1$ then \autoref{symex} yields an automorphism $a\mapsto uau$ of $A$ exchanging $p$ and $q$ so $\norm{(p-q)_+}=\norm{(q-p)_+}$ and hence
\[\norm{p-q}=\norm{(p-q)_+}\vee\norm{(q-p)_+}=\norm{(p-q)_+}=\mathbf{h}(p,q).\qedhere\]
\end{proof}

\begin{cor}\label{dsadp}
On $\mathcal{P}$, $\mathbf{d}=\mathbf{h}=\mathbf{p}=\mathbf{n}$.
\end{cor}

\begin{proof}
If $\mathbf{d}(p,q)<1$ then $qp$ is well-supported so we have a range projection $[qp]=f(qpq)\in A$ (for continuous $f$ on $[0,1]$ that is $1$ on $\sigma(qpq)$).  By \cite{Bice2012} \S2.3,
\begin{align*}
\mathbf{d}(p,q)&=||p-[qp]||\geq\mathbf{p}(p,q)\geq\mathbf{h}(p,q)\quad\text{and}\\
0&=p(q-[qp])=(p-[qp])(q-[qp])\quad\text{so}\\
(p-q)_+&=(p-[qp])_++([qp]-q)_+=(p-[qp])_+\quad\text{and hence}\\
\mathbf{h}(p,q)&=\mathbf{h}(p,[qp])=||p-[qp]||,\quad\text{by \autoref{epqhpq}}.
\end{align*}
While if $\mathbf{d}(p,q)=1$ then $1=||pq^\perp p||=||(pq^\perp p)_+||$ so, by \eqref{aba+},
\[1=||(pq^\perp p)_+||=||(p(p-q)p)_+||\leq||(p-q)_+||=\mathbf{h}(p,q)\leq\mathbf{p}(p,q)\leq||p-q||\leq1.\]
Thus $\mathbf{d}=\mathbf{h}=\mathbf{p}$ on $\mathcal{P}$.  A similar argument with $[p^\perp q^\perp]^\perp\in A$ applies to $\mathbf{n}$.
\end{proof}

Alternatively we could have noted that, by reverting to a C*-subalgebra if necessary, it suffices to prove \autoref{dsadp} when $A$ is generated by $p$ and $q$.  As every irreducible representation of a C*-algebra generated by a pair of projections is on a Hilbert space of dimension at most $2$, for $\mathbf{d}=\mathbf{h}$ it suffices to consider $A=\mathbb{C}$ or $M_2$, which can be done with some elementary calculations.

For our characterizations of $\mathbf{d}$-filters, we will also need to consider a number of unary functions defined from $\mathbf{d}$.  In general, for $\mathbf D\colon X\times X\to [0,\infty]$, we define $x\mathbf{D},\mathbf{D}y:X\rightarrow[0,\infty]$ by fixing the left/right coordinate, i.e.
\[
x\mathbf{D}(y)=\mathbf{D}(x,y)=\mathbf{D}y(x).
\]

\begin{prp}\label{aba=a+b}
For any $a,b\in A^1_+$ and $\epsilon\in(0,1)$,
\[\mathbf{d}(aba)\ \approx\ \mathbf{d}a+\mathbf{d}b\ \approx\ \mathbf{d}(\epsilon a+(1-\epsilon)b).\]
\end{prp}

\begin{proof}  First note $\mathbf{d}aba\leq2\mathbf{d}a+\mathbf{d}b$ and hence $\mathbf{d}(aba)\precapprox\mathbf{d}a+\mathbf{d}b$, as
\begin{align*}
\mathbf{d}(c,aba)&=\norm{c-caba}\\
&\leq\norm{c-ca}+\norm{ca-cba}+\norm{cba-caba}\\
&\leq2\mathbf{d}(c,a)+\mathbf{d}(c,b).
\end{align*}

Next, as $aba\leq a^2\leq a$, $\mathbf{d}(c,a)^2\leq\mathbf{d}(c,aba)+\mathbf{h}(aba,a)=\mathbf{d}(c,aba)$ and
\begin{align*}
\mathbf{d}(c,b)^2 &=\norm{cb^{\perp2}c}\leq\norm{cb^\perp c}=\norm{c^2-cbc}\\
&\leq\norm{c^2-cabac}+\norm{cabac-cbac}+\norm{cbac-cbc}\\
&\leq\norm{c-caba}\norm{c}+\norm{ca-c}\norm{bac}+\norm{cb}\norm{ac-c}\\
&\leq\mathbf{d}(c,aba)+2\mathbf{d}(c,a)\\
&\leq\mathbf{d}(c,aba)+2\sqrt{\mathbf{d}(c,aba)}.
\end{align*}
Thus $\mathbf{d}a+\mathbf{d}b\precapprox\mathbf{d}(aba)$ and hence $\mathbf{d}a+\mathbf{d}b\approx\mathbf{d}(aba)$.

In particular, for any $n\in\mathbb{N}$, setting $a=b$ above and using \eqref{d2<hd} yields
\begin{equation}\label{dan}
\mathbf{d}a\precapprox\mathbf{d}a^n\precapprox\mathbf{d}a^{3^n}\precapprox\mathbf{d}a.
\end{equation}
Also $\sup_{x\in[0,1]}(\epsilon x+1-\epsilon)^n(1-x)\leq\frac{1}{n\epsilon}$ so
\begin{align*}
\mathbf{d}a-\tfrac{1}{n\epsilon}&\leq\mathbf{d}a-\mathbf{d}((\epsilon a+1-\epsilon)^n,a)\\
&\leq\mathbf{d}(\epsilon a+1-\epsilon)^n\qquad\text{by \eqref{tri}}\\
&\precapprox\mathbf{d}(\epsilon a+1-\epsilon)\qquad\text{by \eqref{dan}}\\
&\precapprox\mathbf{d}(\epsilon a+(1-\epsilon)b)\qquad\text{by \eqref{d2<hd}}.
\end{align*}
As $n$ was arbitrary, $\mathbf{d}a\precapprox\mathbf{d}(\epsilon a+(1-\epsilon)b)$ and, by symmetry, $\mathbf{d}b\precapprox\mathbf{d}(\epsilon a+(1-\epsilon)b)$.

On the other hand, $\mathbf{d}(\epsilon a+(1-\epsilon)b)\precapprox\mathbf{d}a+\mathbf{d}b$ follows from
\begin{align*}
\mathbf{d}(c,\epsilon a+(1-\epsilon)b)&=\norm{c-c(\epsilon a+(1-\epsilon)b)}\\
&=\norm{\epsilon c-\epsilon ca+(1-\epsilon)c-(1-\epsilon)cb)}\\
&\leq\epsilon\norm{c-ca}+(1-\epsilon)\norm{c-cb}\\
&=\epsilon\mathbf{d}(c,a)+(1-\epsilon)\mathbf{d}(c,b).\qedhere
\end{align*}
\end{proof}

A slightly better substitute for multiplication on $A^1_+$ than $aba$ might be
\[a\odot b=\sqrt{a}b\sqrt{a}.\]
For if $ab=ba$ then $a\odot b=ab$ and $(a\odot b)\odot c=a\odot(b\odot c)$.  In particular, $\odot$ is left alternative, i.e., $(a\odot a)\odot b=a\odot(a\odot b)$ and also right distributive, i.e., $a\odot(b+c)=a\odot b+a\odot c$.  We can also quantify $\ll$ using $\odot$ by
\[\mathbf{f}(a,b)=\norm{a-a\odot b}.\]
The advantage of $\mathbf{f}$ over $\mathbf{d}$ is that it determines $\mathbf{h}$ in a natural way.

\begin{prp}
We have $\mathbf{d}\approx\mathbf{f}=\mathbf{f}\circ\mathbf{h}$ and
\begin{equation}\label{hf}
\mathbf{h}(a,b)=\sup_{c\in A^1_+}(\mathbf{f}(c,b)-\mathbf{f}(c,a)).
\end{equation}
\end{prp}

\begin{proof}  First note that, for $a,b\in A^1_+$
\[\norm{ab^\perp}^2=\norm{b^\perp a^2b^\perp}\leq\norm{a^2b}.\]
Thus, as binary functions on $A^1_+$,
\[\mathbf{d}(a,b)=\norm{ab^\perp}\approx\norm{\sqrt{a}b^\perp}\approx\norm{\sqrt{a}\sqrt{b^\perp}}\approx\norm{\sqrt{a}\sqrt{b^\perp}}^2=\mathbf{c}(a,b).\]

  As in the proof of \eqref{d2<hd}, we have
\begin{align*}
\mathbf{f}(a,b)&=\norm{a\odot b^\perp}\\
&=\norm{(a\odot b^\perp)_+}\\
&\leq\norm{(a\odot c^\perp)_+}+\norm{(a\odot(b^\perp-c^\perp))_+}\\
&\leq\norm{(a\odot c^\perp)}+\norm{(c-b)_+}\\
&\leq\mathbf{f}(a,c)+\mathbf{h}(c,b).
\end{align*}
Thus $\mathbf{f}=\mathbf{f}\circ\mathbf{h}$ and $\mathbf{h}(a,b)\geq\sup_{d\in A^1_+}(\mathbf{f}(d,a)-\mathbf{f}(d,b))$.  Conversely, take $a,b\in A^1_+$.  If $\mathbf{h}(a,b)=0$ then the reverse inequality is immediate from $\mathbf{h}(a,b)=\mathbf{f}(0,b)-\mathbf{f}(0,c)$.  Otherwise, for any $\epsilon>0$, we can take a pure state $\phi$ on $A$ with $\mathbf{h}(a,b)<\phi(a-b)+\epsilon$.  By \cite[Proposition 2.2]{AkemannAndersonPedersen1986}, we have $c\in A^1_+$ with
\[\phi(c)=1\qquad\text{and}\qquad\norm{c\odot a-c\phi(a)}<\epsilon.\]
Thus $\norm{c\odot a^\perp-c\phi(a^\perp)}<\epsilon$ so $\norm{c\odot a^\perp}<\norm{c\phi(a^\perp)}+\epsilon=\phi(a^\perp)+\epsilon$ and
\begin{align*}
\mathbf{h}(a,b)&<\phi(a-b)+\epsilon\\
&=\phi(b^\perp-a^\perp)+\epsilon\\
&=\phi(c\odot b^\perp)-\phi(a^\perp)+\epsilon\\
&<\norm{c\odot b^\perp}-\norm{c\odot a^\perp}+2\epsilon.\\
&=\mathbf{f}(c,b)-\mathbf{f}(c,a)+2\epsilon.
\end{align*}
As $\epsilon>0$ was arbitrary, we are done.
\end{proof}

The drawback of $\mathbf{f}$ is that it may not be a distance.  Indeed, by \eqref{hf}, $\mathbf{f}$ is a distance if and only if $\mathbf{h}\leq\mathbf{f}$.  But by \autoref{dsadp}, for projections $p,q\in A$,
\[\mathbf{h}(p,q)=\mathbf{d}(p,q)=\sqrt{\norm{pq^\perp p}}=\sqrt{\mathbf{f}(p,q)}.\]
So $\mathbf{h}\nleq\mathbf{f}$ whenver we have projections $p,q\in A$ with $0<\mathbf{f}(p,q)<1$.  A more involved argument could even show that $\mathbf{h}\nleq\mathbf{f}$ whenever $A$ is non-commutative.

\section{Filters}\label{F}

\begin{dfn}
For $\mathbf{D}\colon X\times X\to [0,\infty]$, define the following conditions on $Y\subseteq X$.
\begin{align}
\label{Dfilter}\tag{$\mathbf{D}$-filter}a,b\in Y\quad&\Leftrightarrow\quad\inf_{c\in Y}(\mathbf{D}(c,b)+\mathbf{D}(c,a))=0.\\
\label{Ddirected}\tag{$\mathbf{D}$-directed}a,b\in Y\quad&\Rightarrow\quad\inf_{c\in Y}(\mathbf{D}(c,b)+\mathbf{D}(c,a))=0.\\
\label{Dclosed}\tag{$\mathbf{D}$-closed}b\in Y\quad&\Leftarrow\quad\inf_{c\in Y}\mathbf{D}(c,b)=0.\\
\label{Ddominated}\tag{$\mathbf{D}$-initial}b\in Y\quad&\Rightarrow\quad\inf_{c\in Y}\mathbf{D}(c,b)=0.\\
\label{Dcoinitial}\tag{$\mathbf{D}$-coinitial}b\in X\quad&\Rightarrow\quad\inf_{c\in Y}\mathbf{D}(c,b)=0.\\
\label{Dcofinal}\tag{$\mathbf{D}$-cofinal}c\in X\quad&\Rightarrow\quad\inf_{b\in Y}\mathbf{D}(c,b)=0.
\end{align}
For any operation $\bullet:X^n\rightarrow X$, we also call $Y\subseteq X$ \emph{$\bullet$-closed} if $\bullet[Y^n]\subseteq Y$.
\end{dfn}

This terminology covers a number of familiar concepts from metric, order and $\mathrm{C}^*$-algebra theory.  For example, $\mathbf{d}$-cofinal $\geq$-directed subsets of $A^1_+$ are increasing approximate units in the usual sense, when considered as self-indexed nets.  

If $\mathbf{D}$ is a metric, $\mathbf{D}$-coinitial/cofinal means dense while $\mathbf{D}$-closed means closed, with respect to the usual ball topology defined by $\mathbf{D}$. The other terms become trivial \textendash\, specifically, arbitrary subsets are $\mathbf{D}$-initial, while the empty and one-point subsets are the only $\mathbf{D}$-directed subsets.  In particular, for $\mathrm{C}^*$-algebras, $\mathbf{e}$-closed/coinitial means norm closed/dense in the usual sense.  

On the other hand, for any order relation $\leq$ (again identified with its characteristic function), $\leq$-closed means upwards closed, $\leq$-directed means downwards directed and $\leq$-cofinal means cofinal in the usual sense.  In particular, $\leq$-filters are the usual order-theoretic filters and, more generally,
\[
\mathbf{D}\text{-filter}\quad\Leftrightarrow\quad\mathbf{D}\text{-directed and $\mathbf{D}$-closed}.
\]

We now characterize the $\mathrm{C}^*$-algebra filters considered in \cite{Bice2011} in various ways.  First we recall that $F\subseteq A^1_+$ is a \emph{norm filter}, according to \cite{Bice2011} Definition 3.1, if
\[\inf_{\substack{k\in\mathbb{N}\\ a_1,\ldots,a_k\in F}}\norm{a_1a_2\ldots a_kb^\perp}=0\quad\Rightarrow\quad b\in F.\]
Also recall that a subset $C$ of a vector space $X$ is \emph{convex} if $\epsilon a+(1-\epsilon)b\in C$ whenever $a,b\in C$ and $\epsilon\in(0,1)$ and $F\subseteq C$ is a \emph{face} of $C$ if converse also holds, i.e. if, for all $a,b\in C$ and $\epsilon\in(0,1)$,
\[a,b\in F\quad\Leftrightarrow\quad\epsilon a+(1-\epsilon)b\in F\]
(for faces it actually suffices to take $\epsilon=\frac{1}{2}$ or any other fixed element of $(0,1)$).

\begin{thm}\label{FilterEqvs}
For $F\subseteq A^1_+$, the following are equivalent.
\begin{enumerate}
\item\label{dd} $F$ is a $\mathbf{d}$-filter.
\item\label{dh} $F$ is a $\mathbf{d}$-initial $\mathbf{h}$-filter.
\item\label{<=fil} $F$ is the norm closure of a $\mathbf{d}$-initial $\leq$-filter.
\item\label{prodfil} $F$ is norm closed, $\leq$-closed and $\odot$-closed.\footnote{Note \eqref{prodfil}$\Rightarrow$\eqref{normfil} eliminates the real rank zero hypothesis from \cite{Bice2011} Proposition 3.5.}
\item\label{convex} $F$ is norm closed, $\leq$-closed, $^2$-closed and convex.
\item\label{face} $F$ is a norm closed $\mathbf{d}$-cofinal face.
\item\label{normfil} $F$ is a norm filter.

\setcounter{cnt}{\theenumi}
\end{enumerate}
If $A$ is separable or commutative, they are also equivalent to the following.
\begin{enumerate}
\setcounter{enumi}{\thecnt}
\item\label{llllfilter} $F$ is the norm closure of a $\ll$-filter.
\end{enumerate}
\end{thm}

\begin{proof}\
\begin{itemize}
\item[\eqref{dd}$\Leftrightarrow$\eqref{dh}]  As $\mathbf{h}\precapprox\mathbf{d}$,
\begin{align*}
\text{$\mathbf{d}$-directed}\quad&\Rightarrow\quad\text{$\mathbf{h}$-directed}.\\
\text{$\mathbf{h}$-closed}\quad&\Rightarrow\quad\text{$\mathbf{d}$-closed}.
\end{align*}
If $F$ is $\mathbf{d}$-initial then
\[\inf_{b\in F}\mathbf{h}(b,a)=\inf_{b\in F}\inf_{c\in F}(\mathbf{d}(c,b)+\mathbf{h}(b,a))=\inf_{c\in F}(\mathbf{d}\circ\mathbf{h})(c,a).\]
So $0=\inf_{b\in F}\mathbf{h}(b,a)$ implies $0=\inf_{b\in F}\mathbf{d}(b,a)$, as $\mathbf{d}\precapprox\mathbf{d}\circ\mathbf{h}$, which yields the converse implications.

\item[\eqref{dd}$\Rightarrow$\eqref{convex}]  As $\mathbf{d}a^2\precapprox\mathbf{d}a$, $\mathbf{d}(\epsilon a+(1-\epsilon)b)\precapprox\mathbf{d}a+\mathbf{d}b$, and $\mathbf{h}\leq\mathbf{e}$ quantifies $\leq$,
\begin{align*}
\text{$\mathbf{d}$-filter}\quad&\Rightarrow\quad\text{$^2$-closed and convex}.\\
\label{norm<=closed}\text{$\mathbf{h}$-closed}\quad&\Rightarrow\quad\text{norm closed and $\leq$-closed}.
\end{align*}
Thus \eqref{dd}$\Rightarrow$\eqref{convex} now follows from \eqref{dd}$\Rightarrow$\eqref{dh}.

\item[\eqref{convex}$\Rightarrow$\eqref{dd}]  For any $a\in A^1_+$, $\mathbf{d}(a^{2^n},a)\rightarrow0$ as
\begin{equation}\label{a^na}
\mathbf{d}(a^n,a)=\norm{a^na^\perp}\leq\sup_{x\in[0,1]}x^n(1-x)\leq1/n.
\end{equation}
Taking $a=\frac{1}{2}(b+c)$, $\mathbf{d}b+\mathbf{d}c\precapprox\mathbf{d}a$, $\mathbf{d}(a^{2^n},a)\rightarrow0$ and $\mathbf{h}\precapprox\mathbf{n}$ yields
\begin{align*}
\text{$^2$-closed and convex}\quad&\Rightarrow\quad\text{$\mathbf{d}$-directed}.\\
\text{norm closed and $\leq$-closed}\quad&\Rightarrow\quad\text{$\mathbf{h}$-closed}.
\end{align*}
As $\mathbf{h}\precapprox\mathbf{d}$, $\mathbf{h}$-closed implies $\mathbf{d}$-closed.

\item[\eqref{dd}$\Rightarrow$\eqref{prodfil}]  As $\mathbf{d}(a\odot b)\precapprox\mathbf{d}\sqrt{a}+\mathbf{d}b\precapprox\mathbf{d}a+\mathbf{d}b$,
\[\text{$\mathbf{d}$-filter}\quad\Rightarrow\quad\text{$\odot$-closed}.\]
By \eqref{dd}$\Rightarrow$\eqref{dh}, any $\mathbf{d}$-filter is $\mathbf{h}$-closed and hence norm closed and $\leq$-closed.

\item[\eqref{prodfil}$\Rightarrow$\eqref{dd}]  Taking $c=a\odot b$, $\mathbf{d}a+\mathbf{d}b\precapprox\mathbf{d}\sqrt{a}+\mathbf{d}b\precapprox\mathbf{d}c$ and $\mathbf{d}(c^{2^n},c)\rightarrow0$ yields
\[\text{$\odot$-closed}\quad\Rightarrow\quad\text{$\mathbf{d}$-directed}.\]
As $\mathbf{d}\precapprox\mathbf{h}\precapprox\mathbf{n}$, norm closed and $\leq$-closed implies $\mathbf{d}$-closed.

\item[\eqref{normfil}$\Rightarrow$\eqref{dd}]  We immediately see that
\[\text{norm filter}\quad\Rightarrow\quad\text{$\mathbf{d}$-closed}.\]
For any $a,b\in F$, $\norm{(aba)^n(aba)^\perp}\rightarrow0$, by \eqref{a^na}, so $aba\in F$.  Thus $(aba)^{3^n}\in F$ too so, as $\mathbf{d}a+\mathbf{d}b\precapprox\mathbf{d}(aba)$,
\[\text{norm filter}\quad\Rightarrow\quad\text{$\mathbf{d}$-directed}.\]

\item[\eqref{dd}$\Rightarrow$\eqref{normfil}]  Assume $F\subseteq A^1_+$ is a $\mathbf{d}$-filter and take $b\in A^1_+$ with
\[\inf_{\substack{k\in\mathbb{N}\\ a_1,\ldots,a_k\in F}}\norm{a_1a_2\ldots a_kb^\perp}=0.\]
As $F$ is $\mathbf{d}$-directed, for any $a_1,\ldots,a_k\in F$ and $\epsilon>0$, we can find $a\in F$ with $\mathbf{d}(a,a_j)\leq\epsilon$, for all $j\leq k$, and hence
\begin{align*}
\mathbf{d}(a,b)&=\norm{a(a_k^\perp+a_k)b^\perp}\\
&\leq\norm{aa_k^\perp}+\norm{aa_kb^\perp}\\
&\leq\epsilon+\norm{a(a_{k-1}^\perp+a_{k-1})a_kb^\perp}\leq\cdots\\
&\leq k\epsilon+\norm{a_1a_2\ldots a_kb^\perp}.
\end{align*}
Thus $\inf_{a\in F}\mathbf{d}(a,b)=0$.  As $F$ is $\mathbf{d}$-closed, $b\in F$ so $F$ is a norm filter.

\item[\eqref{face}$\Rightarrow$\eqref{prodfil}]  Assume $F$ is a norm closed face of $A^1_+$.  We first claim
\[a,b\in F\quad\Rightarrow\quad[a,b]\subseteq F,\]
where $[a,b]=\{c\in A^1_+:a\leq c\leq b\}$.  For if $c\in[a,b]$ then $d=a+b-c\in A^1_+$ and $\frac{1}{2}(c+d)=\frac{1}{2}(a+b)\in F$ and hence $c\in F$.  Next we claim that
\[a\in F\quad\Rightarrow\quad f(a)\subseteq F,\]
for any continuous $f$ on $[0,1]$ taking $0$ to $0$ and $1$ to $1$, as in \cite[Lemma 2.1]{AkemannPedersen1992}.  For let $g(a)=(2a-1)_+$ and $h(a)=f(a^\perp)^\perp$.  Then $\frac{1}{2}(g(a)+h(a))=a$ so $g(a),h(a)\in F$.  Thus $g^n(a),h^n(a)\in F$, for all $n$, and hence
\[f(a)\in\overline{\bigcup[g^n(a),h^n(a)]}\subseteq F.\]
Now we claim $F$ is $\odot$-closed, as in the proof of \cite[Theorem 2.9]{AkemannPedersen1992}.  For given $a,b\in F$ and $\epsilon\in(0,1)$, let $c_\epsilon=(\epsilon\sqrt{a}+(1-\epsilon)b)\in F$ so
\[\epsilon c_\epsilon\sqrt{a}c_\epsilon+(1-\epsilon)c_\epsilon bc_\epsilon=c_\epsilon^3\in F\]
and hence $c_\epsilon bc_\epsilon\in F$.  As $\epsilon\rightarrow1$, $c_\epsilon bc_\epsilon\rightarrow a\odot b\in F$, as required.

To see that $F$ is $\leq$-closed when $F$ is also $\mathbf{d}$-cofinal, take $a\in A^1_+$ with $a\geq b\in F$ and $(c_n)\subseteq F$ with $\mathbf{d}(a,c_n)\rightarrow0$.  Then $c_n\odot b\in F$ and $c_n\odot b\leq c_n\odot a\leq c_n$ so $c_n\odot a\rightarrow a\in F$, by our first claim above.

\item[\eqref{dd}$\Rightarrow$\eqref{face}]  Take a $\mathbf{d}$-filter $F$. By \eqref{dd}$\Rightarrow$\eqref{convex}, $F$ is norm closed and convex.  As $F$ is $\mathbf{d}$-initial, $\mathbf{d}$-closed and $\mathbf{d}a+\mathbf{d}b\precapprox\mathbf{d}(\epsilon a+(1-\epsilon)b)$, $F$ is also a face of $A^1_+$.  As $A$ has an approximate unit, for any $a,b\in A^1_+$ and $\epsilon>0$ we have $c\in A^1_+$ with $\mathbf{d}(a,c),\mathbf{d}(b,c)<\epsilon$.  In particular this applies to $b\in F$ and, as $\mathbf{n}\precapprox\mathbf{h}\precapprox\mathbf{d}$, we may take $c\geq b$.  As $F$ is $\leq$-closed, $c\in F$ which shows that $F$ is $\mathbf{d}$-cofinal.

\item[\eqref{<=fil}$\Rightarrow$\eqref{dh}]  As $\mathbf{d}$ and $\mathbf{h}$ are $\mathbf{e}$-invariant and $\mathbf{n}\precapprox\mathbf{h}$, for any $F\subseteq A^1_+$,
\begin{align*}
F\text{ is $\mathbf{d}$-initial}\quad&\Rightarrow\quad\overline{F}\text{ is $\mathbf{d}$-initial}.\\
F\text{ is $\leq$-directed}\quad&\Rightarrow\quad\overline{F}\text{ is $\mathbf{h}$-directed}.\\
F\text{ is $\leq$-closed}\quad&\Rightarrow\quad\overline{F}\text{ is $\mathbf{h}$-closed}.
\end{align*}

\item[\eqref{dd}$\Rightarrow$\eqref{<=fil}]  Take a $\mathbf{d}$-filter $F$ and assume first that $A$ is unital.  Consider the invertible elements $G$ of $F$.  For every $a\in F$ and $\epsilon>0$, $(1-\epsilon)a+\epsilon\in G$ so $\overline{G}=F$.  As $F$ is $\leq$-closed and $\mathbf{d}$-initial, so is $G$.  It only remains to show that $G$ is $\leq$-directed.  So take $a,b\in G$.  For some $\epsilon>0$ and $a',b'\in A^1_+$, we have $a=\epsilon+(1-\epsilon)a'$ and $b=\epsilon+(1-\epsilon)b'$.  As $F$ is a face of $A^1_+$ containing $1$, $a',b'\in F$.  As $F$ is $\mathbf{h}$-directed, we have $c'\in F$ with $\mathbf{h}(c',a'),\mathbf{h}(c',b')<\epsilon/2$.  Letting $c=\epsilon+(1-\epsilon)c'\in F$, we thus have $\mathbf{h}(c,a),\mathbf{h}(c,b)<\epsilon/2$.  Thus $d=c-(c-a)_+-(c-b)_+$ is an invertible element of $A^1_+$ (as $c\geq\epsilon$ and $\norm{c-d}<\epsilon$) and we further claim that $d\in F$ and hence $d\in G$.  Indeed, for any $\delta>0$, we have $e\in F$ with $\mathbf{d}(e,a),\mathbf{d}(e,b),\mathbf{d}(e,c)<\delta$.  Thus $\norm{e(c-a)_+}\leq\norm{e(c-a)}<2\delta$ and $\norm{e(c-b)_+}<2\delta$ so $\mathbf{d}(e,d)<5\delta$.  As $\delta>0$ was arbitrary and $F$ is $\mathbf{d}$-closed, we have $d\in F$, as required.

If $A$ is not unital then first extend $F$ to a $\mathbf{d}$-filter $F'$ in $\widetilde{A}^1_+$ by taking the (upwards) $\leq$-closure.  By what we just proved, the invertible elements $G'$ of $F'$ are a $\leq$-filter with $\overline{G'}=F'$.  In particular, $G'$ is $\mathbf{d}$-coinitial in $F'$.  Thus, for any $a\in F$, $aG'a$ is $\mathbf{d}$-coinitial in $F$.  Hence the $\leq$-closure $G$ of $aG'a$ in $A^1_+$ is a $\mathbf{d}$-initial $\leq$-filter in $A^1_+$ with $\overline{G}=F$, by \autoref{h=n}.

\item[\eqref{dd}$\Rightarrow$\eqref{llllfilter}]  If $A$ is separable then we can take dense $(a_n)\subseteq F$ and let $a=\sum2^{-n}a_n\in F$.  As noted above, $f(a)\in F$, for any continuous $f$ on $[0,1]$ taking $0$ to $0$ and $1$ to $1$.  Thus by choosing such $(f_n)$ converging pointwise to $0$ everywhere except at $1$ and satisfying $f_1\gg f_2\gg\ldots$, the (upwards) $\ll$-closure $G$ of $(f_n(a))$ is a $\ll$-filter with $F=\overline{G}$.

Now take
\[G=\{a\in F:a\gg b\in F\}.\]
Again, if $a\in F$ then $f(a)\in F$, for any continuous $f$ on $[0,1]$ taking $0$ to $0$ and $1$ to $1$.  In particular, for any $\epsilon>0$ we can take $f(x)=(1+\epsilon)x\wedge1$ and $g(x)=(\epsilon^{-1}(x-1)+1)_+$ so $f(a)\gg g(a)\in F$ and hence $f(a)\in G$.  As $\epsilon\rightarrow0$, $f(a)\rightarrow a$ so $F=\overline{G}$.  Likewise, if $a\gg b\in F$ then $a\gg f(b)\gg g(b)\in F$ so
\[G=\{a\in F:a\gg b\gg c\in F\}.\]
If $A$ is commutative, $a\gg b$ and $a'\gg b'$ implies $aa'\gg bb'$.  For $a,a'\in G$, we have $b,b',c,c'\in F$ with $a\gg b\gg c$ and $a'\gg b'\gg c'$ so $bb'\gg cc'\in F$ and hence $a,a'\gg bb'\in G$, i.e. $G$ is $\ll$-directed and hence a $\ll$-filter.

\item[\eqref{llllfilter}$\Rightarrow$\eqref{<=fil}]  As $\mathbf{d}$ quantifies $\ll$ and $\ll\circ\leq\ \ \subseteq\ \ \ll\ \ \subseteq\ \ \leq$,
\begin{align*}
\text{$\ll$-initial}\quad&\Rightarrow\quad\text{$\mathbf{d}$-initial}.\\
\text{$\ll$-closed and $\ll$-initial}\quad&\Rightarrow\quad\text{$\leq$-closed}.\\
\text{$\ll$-directed}\quad&\Rightarrow\quad\text{$\leq$-directed}.\qedhere
\end{align*}
\end{itemize}
\end{proof}

In \eqref{prodfil} and \eqref{convex}, we could not replace `$\leq$-closed' with `$\ll$-closed'.  For example, the norm closure $C$ of the convex combinations of the functions $x^n$ in $C([0,1])$, for $n\in\mathbb{N}$, satisfies these conditions \textendash\, as every $f\in C$ is positive on $(0,1]$, $C$ is vacuously $\ll$-closed \textendash\, however $C$ is not $\leq$-closed, being bounded above by $x$.  Although we could replace `norm closed and $\leq$-closed' with `$\mathbf{h}$-closed' or `$\mathbf{d}$-closed'.

Furthermore, not every $\mathbf{d}$-filter is the norm closure of a $\ll$-filter.  Indeed, if this were the case then, for any non-unital $A$, the $\mathbf{d}$-filter $\{1-a:a\in A^1_+\}$ in $\widetilde{A}^1_+$ would be the norm closure of a $\ll$-filter $F$.  Then $1-F$ would be an `almost idempotent' approximate unit of $A$ in the sense of \cite[II.4.1.1]{Blackadar2017}.  However, a $\mathrm{C}^*$-algebra was recently constructed in \cite{BiceKoszmider2017} that does not possess such an approximate unit.  All $\omega_1$-unital $\mathrm{C}^*$-algebras have them though, by another result from \cite{BiceKoszmider2017}, so \eqref{llllfilter} could be extended to any $A$ with a dense subset of size $\leq\omega_1$.

We can at least say a bit more in the commutative case.

\begin{prp}
If $A$ is commutative and $F\subseteq A^1_+$ is a $\mathbf{d}$-filter then
\[G=\{a\in F:a\gg b\in F\}\]
is the unique $\ll$-filter with $F=\overline{G}$.
\end{prp}

\begin{proof}
The only thing left to show is uniqueness.  By the Gelfand represtentation, we may assume that $A=C_0(X)$ for some locally compact Hausdorff $X$, so
\[f\ll g\qquad\Leftrightarrow\qquad X\setminus f^{-1}\{0\}\subseteq g^{-1}\{1\}.\]
Take a $\mathbf{d}$-filter $F\subseteq A^1_+$ and let
\[C=\bigcap_{f\in F}f^{-1}\{1\}.\]
For any $\ll$-filter $G$ with $\overline{G}=F$, we must also have $C=\bigcap_{g\in G}g^{-1}\{1\}$.  Otherwise, we could pick some $x\in\bigcap_{g\in G}g^{-1}\{1\}\setminus C$ and $f\in F$ with $f(x)\neq1$ and then $\norm{f-g}\geq g(x)-f(x)=1-f(x)>0$, for all $g\in G$, contradicting $\overline{G}=F$.  

Take $f\in A^1_+$ with $C\subseteq f^{-1}\{1\}^\circ$.  For every $x\in X\setminus f^{-1}\{1\}^\circ$, we have $g_x\in G$ with $g_x(x)\neq1$.  Thus we can pick arbitrary $g\in G$ and cover the compact set $g^{-1}[\frac{1}{2},1]\setminus f^{-1}\{1\}^\circ$ with finitely many open sets $X\setminus g_{x_1}^{-1}\{1\},\ldots,X\setminus g_{x_k}^{-1}\{1\}$.  As $G$ is $\ll$-directed, we have some $h\in G$ with $h\ll g,g_{x_1},\ldots,g_{x_k}$ and hence
\[X\setminus h^{-1}\{0\}\subseteq g^{-1}\{1\}\cap g_{x_1}^{-1}\{1\}\cap\ldots\cap g_{x_k}^{-1}\{1\}\subseteq f^{-1}\{1\}^\circ\subseteq f^{-1}\{1\},\]
i.e., $h\ll f$.  Thus $f\in G$, as $G$ is $\ll$-closed, so
\[\{f\in F:f\gg g\in F\}\subseteq\{f\in A^1_+:C\subseteq f^{-1}\{1\}^\circ\}\subseteq G.\]
Conversely, $G\subseteq\{f\in F:f\gg g\in F\}$, as $G$ is a $\ll$-filter contained in $F$.
\end{proof}

This does not extend to non-commutative $A$, i.e.
\[G=\{a\in F:a\gg b\in F\}\]
may fail to be a $\ll$-filter and $F$ may contain various dense $\ll$-filters.  For example, consider $A=C([0,1],M_2)$ and take everywhere rank 1 projections $p,q\in A$ with $p(0)=P=q(0)$ but $p(x)\neq q(x)$, for all $x>0$.  Also take continuous $f_n$ on $[0,1]$ with $f_1\gg f_2\gg\ldots$ and $\bigcap_n f^{-1}\{1\}=\{0\}$.  Then the $\ll$-closures $F,G\subseteq A^1_+$ of $(f_np)$ and $(f_nq)$ respectively are distinct $\ll$-filters with $\overline{F}=\overline{G}=\{a\in A^1_+:a(0)\geq P\}$.

\begin{dfn}
We say $Y$ \emph{generates} a $\mathbf{D}$-filter $F\subseteq X$ if $F$ is the smallest $\mathbf{D}$-filter containing $Y$.  We call $X$ a \emph{$\mathbf{D}$-semilattice} if every $Y\subseteq X$ generates a $\mathbf{D}$-filter.
\end{dfn}

For posets, $\leq$-semilattices are precisely the meet semilattices in the usual sense.

\begin{prp}
If $\leq$ is a partial order on $X$ then
\[X\text{ is a $\leq$-semilattice}\qquad\Leftrightarrow\qquad\text{every $x,y\in X$ has an infimum }x\wedge y\in X.\]
\end{prp}

\begin{proof}
If every $x,y\in X$ has an infimum $x\wedge y\in X$ then the $\leq$-closure of the $\wedge$-closure of any $Y\subseteq X$ is the $\leq$-filter generated by $Y$, so $X$ is a $\leq$-semilattice.  If some $x,y\in X$ have no infimum then
\[\bigcap_{z\leq x,y}\{w\in X:z\leq w\}\]
is an intersection of $\leq$-filters containing $x$ and $y$ but no lower bound of $x$ and $y$.  Thus $\{x,y\}$ does not generate a $\leq$-filter and hence $X$ is not a $\leq$-semilattice.
\end{proof}

\begin{prp}\label{dsemilat}
$A^1_+$ is a $\mathbf{d}$-semilattice.
\end{prp}

\begin{proof}
For any $B\subseteq A^1_+$, let $D$ be the $\odot$-closure of $B$, so $D$ is $\mathbf{d}$-directed and every $\mathbf{d}$-filter containing $B$ must contain $D$.  Let $F$ by the $\mathbf{d}$-closure of $D$, so $F$ is a $\mathbf{d}$-filter and every $\mathbf{d}$-filter containing $B$ and hence $D$ contains $F$, i.e., $F$ is the $\mathbf{d}$-filter generated by $B$.
\end{proof}

\begin{dfn}
We call $Y\subseteq X$ \emph{$\mathbf{D}$-centred} if, for all $y_1,\ldots,y_k\subseteq Y$,
\[\inf_{x\in X}\mathbf{D}(x,y_1)+\ldots+\mathbf{D}(x,y_k)=0.\]
\end{dfn}

If $\leq$ is a partial order on $X$ then $Y\subseteq X$ is $\leq$-centred if and only if every finite subset of $Y$ has a lower bound in $X$, i.e., if and only if $Y$ is centred in the usual order theoretic sense.

As with filters, we see that the centred subsets of $\mathrm{C}^*$-algebras considered in \cite{Bice2011} are precisely the $\mathbf{d}$-centred subsets, this time in the positive unit sphere $A^{=1}_+$ rather than the positive unit ball $A^1_+$. Specifically, recall that $C\subseteq A^{=1}_+$ is \emph{norm centred}, according to \cite[Definition 2.1]{Bice2011}, if the multiplicative closure of $C$ is contained in the unit sphere.

\begin{prp}
For $C\subseteq A^{=1}_+$, the following are equivalent.
\begin{enumerate}
\item\label{dcen} $C$ is $\mathbf{d}$-centred in $A^{=1}_+$.
\item\label{hcen} $C$ is $\mathbf{h}$-centred in $A^{=1}_+$.
\item\label{ncen} $C$ is norm centred.
\item\label{pfil} $C$ is generates a proper $\mathbf{d}$-filter in $A^1_+$.
\end{enumerate}
\end{prp}

\begin{proof}\
\begin{itemize}
\item[\eqref{dcen}$\Rightarrow$\eqref{hcen}]  Immediate from $\mathbf{h}\leq2\mathbf{d}$.
\item[\eqref{hcen}$\Rightarrow$\eqref{dcen}]  If $\norm{a}=1$ then $\norm{a^n}=1$ and, for any $b\in A^1_+$,
\[\mathbf{d}(a^n,b)^2\leq\mathbf{d}(a^n,a)+\mathbf{h}(a,b)\rightarrow\mathbf{h}(a,b).\]

\item[\eqref{dcen}$\Rightarrow$\eqref{ncen}]  If $C$ is $\mathbf{d}$-centred then, for any $a_1,\ldots,a_k\in C$ and $\epsilon>0$, we have $b\in A^{=1}_+$ with $\mathbf{d}(b,a_1)+\ldots+\mathbf{d}(b,a_k)<\epsilon$ so
\begin{align*}
1&=\norm{b(a_1+a_1^\perp)}\\
&\leq\norm{b(a_2+a_2^\perp)a_1}+\mathbf{d}(b,a_1)\\
&\leq\norm{b(a_3+a_3^\perp)a_2a_1}+\mathbf{d}(b,a_2)+\mathbf{d}(b,a_1)\leq\ldots\\
&\leq\norm{ba_k\ldots a_1}+\epsilon\\
&\leq\norm{a_k\ldots a_1}+\epsilon.
\end{align*}
As $\epsilon>0$ was arbitrary, $C$ is norm centred.

\item[\eqref{ncen}$\Rightarrow$\eqref{pfil}]  If the multiplicative closure of $C$ is contained in the unit sphere then the same goes for the closure $D$ of $C$ under the operation $(a,b)\mapsto aba$.  The same then applies to the $\mathbf{d}$-closure $F$ of $D$ so, in particular, $F$ is proper.  As in \autoref{dsemilat}, $F$ is the $\mathbf{d}$-filter generated by $C$.

\item[\eqref{pfil}$\Rightarrow$\eqref{dcen}]  We show that a $\mathbf{d}$-filter $F\subseteq A^1_+$ is proper if and only if it is contained in the positive unit sphere. For if $a\in F$ and $\norm{a}<1$ then $a^n\in F$ so, for any $b\in A^1_+$,
\[\mathbf{d}(a^n,b)\leq\norm{a^n}=\norm{a}^n\rightarrow0,\]
and hence $b\in F$, as $F$ is $\mathbf{d}$-closed.

If $C$ is contained in such a $\mathbf{d}$-filter $F$ then, for any $c_1,\ldots,c_n\in C$,
\[\inf_{a\in A^{=1}_+}\mathbf{d}(a,c_1)+\ldots+\mathbf{d}(a,c_k)\leq\inf_{a\in F}\mathbf{d}(a,c_1)+\ldots+\mathbf{d}(a,c_k)=0,\]
as $F$ is $\mathbf{d}$-directed, i.e., $C$ is $\mathbf{d}$-centred in $A^{=1}_+$.\qedhere
\end{itemize}
\end{proof}

In particular, the maximal $\mathbf{d}$-centred subsets are precisely the maximal proper $\mathbf{d}$-filters.  These were the original quantum filters defined by Farah and Weaver to study pure states.  Pure states correspond to minimal projections in $A^{**}$ and, more generally, $\mathbf{d}$-filters correspond to the compact projections $A^{**}$ introduced by Akemann (which was touched on briefly in \cite[Corollary 3.4]{Bice2011}).  This is the connection we explore next.

\section{Compact Projections}\label{CP}

Let $\uparrow\!p$ denote the upper set in $A^1_+$ defined by any projection $p\in A^{**}$, that is,
\[\uparrow\!p=\{a\in A^1_+:p\leq a\}.\]

\begin{dfn}
A projection $p\in A^{**}$ is \emph{compact} if $p={\displaystyle\bigwedge}\uparrow\!p$.
\end{dfn}

Note that for $p$ to be compact it is implicit that $\uparrow\!p$ is non-empty.

\begin{thm}\label{CompactFilter}
We have mutually inverse bijections
\[p\mapsto\ \uparrow\!p\qquad\text{and}\qquad F\mapsto\bigwedge F\]
between compact projections $p\in A^{**}$ and $\mathbf{d}$-filters $F\subseteq A^1_+$.  Moreover, for compact projections $p,q\in A^{**}$ and corresponding $\mathbf{d}$-filters $F,G\subseteq A^1_+$,
\begin{equation}\label{supinfd}
\mathbf{d}(p,q)=\sup_{b\in G}\inf_{a\in F}\mathbf{d}(a,b).
\end{equation}
\end{thm}

\begin{proof}
Take a projection $p\in A^{**}$ and consider $\uparrow\!p$.  If $p=pa$ then $pa^2=pa=p$, i.e., $p\ll a$ implies $p\ll a^2$ so $\uparrow\!p$ is $^2$-closed.  Likewise, if $p=pa$, $p=pb$ and $\epsilon\in(0,1)$ then $p(\epsilon a+(1-\epsilon)b)=\epsilon p+(1-\epsilon)p=p$, i.e., $p\ll(\epsilon a+(1-\epsilon)b)$ so $\uparrow\!p$ is convex.  Also $p\ll a\leq b\in A^1_+$ implies $p\ll b$, as $\mathbf{d}^2\leq\mathbf{d}\circ\mathbf{h}$ on $A^1_+$, so $\uparrow\!p$ is $\leq$-closed.  Finally, if $a_n\rightarrow a$ and $p\ll a_n$, for all $n$, then $\mathbf{d}(p,a)=\lim\mathbf{d}(p,a_n)=0$, i.e., $p\ll a$ so $\uparrow\!p$ is norm closed and thus a $\mathbf{d}$-filter, by \autoref{FilterEqvs}.

Conversely, take a $\mathbf{d}$-filter $F\subseteq A^1_+$ which, by \eqref{FilterEqvs}, contains a dense $\leq$-filter $F'$.  The pointwise infimum of $F'$ on $\mathsf{Q}$ (recall that $\mathsf{Q}$ is the space of quasistates on $A$) is an affine function and thus defines an element $p\in A^{**}$.  As $\leq$ on $A^{**}_\mathrm{sa}$ is determined by $\mathsf{Q}$, $p=\bigwedge F'=\bigwedge F=\bigwedge\uparrow\!p$.  As $p$ takes $\mathsf{Q}$ to $[0,1]$, $p$ is positive and has norm at most $1$, i.e., $p\in A^{**1}_+$.  As $F$ is $\mathbf{d}$-initial,
\[\sup_{a\in F}\mathbf{d}(p,a)^2\leq\sup_{a\in F}\inf_{b\in F}(\mathbf{h}(p,b)+\mathbf{d}(b,a))=0,\]
i.e., for all $a\in F$, $p\ll a$ so $\sqrt{p}\ll a$.  Thus $\sqrt{p}\leq\bigwedge F=p$ so $p$ is a projection.

Now take another $\mathbf{d}$-filter $G$ containing a dense $\leq$-filter $G'$ and defining a compact projection $q=\bigwedge G$ which is a pointwise infimum of $G$ on $\mathsf{Q}$.  Then $pG'p$ is also $\leq$-directed so $pqp=\bigwedge pG'p=\bigwedge pGp$ is also a pointwise infimum on $\mathsf{Q}$ and hence
\begin{align}
\nonumber\mathbf{d}(p,q)^2&=\norm{pq^\perp p}=\sup_{\phi\in\mathsf{Q}}\phi(pq^\perp p)\\
\nonumber&=\phi(q)-\inf_{\phi\in\mathsf{Q}}\phi(pqp)\\
\nonumber&=\phi(q)-\inf_{\phi\in\mathsf{Q},b\in G}\phi(pbp)\\
\nonumber&=\sup_{\phi\in\mathsf{Q},b\in G}\phi(pb^\perp p)\\
\label{bperp2}&=\sup_{\phi\in\mathsf{Q},b\in G}\phi(pb^{\perp2}p)\\
\nonumber&=\sup_{b\in G}\norm{pb^{\perp2}p}=\sup_{b\in G}\mathbf{d}(p,b)^2.
\end{align}
For \eqref{bperp2}, note that $\norm{ab^\perp}^2=\norm{ab^{\perp2}a}\leq\norm{ab^{\perp2}}\leq\norm{ab^\perp}$ so $\mathbf{d}b\approx\mathbf{d}b^{\perp2\perp}$.  Thus, as $G$ is $\mathbf{d}$-initial and $\mathbf{d}$-closed, $b\in G$ if and only if $b^{\perp2\perp}\in G$, i.e., $\{b^\perp:b\in G\}=\{b^{\perp2}:b\in G\}$.

Fix $b\in G$ and define weak* continuous $f_a:\mathsf{Q}\rightarrow[0,1]$ by
\[f_a(\phi)=(\phi(b^\perp ab^\perp)-\norm{b^\perp pb^\perp})_+.\]
Then $(f_a)_{a\in F'}$ is downwards directed in the product ordering on $[0,1]^\mathsf{Q}$ and converges to $0$ pointwise.  As $\mathsf{Q}$ is weak* compact, Dini's theorem says $(f_a)_{a\in F'}$ must actually converge uniformly to $0$ on $\mathsf{Q}$ and hence
\[\inf_{a\in F'}\norm{b^\perp ab^\perp}\leq\norm{b^\perp pb^\perp}\leq\inf_{a\in F}\norm{b^\perp ab^\perp}.\]
As $F$ is $^2$-closed and $\sqrt{\ }$-closed (as $F$ is $\leq$-closed),
\[\inf_{a\in F}\mathbf{d}(a,b)^2=\inf_{a\in F}\norm{b^\perp a^2b^\perp}=\inf_{a\in F}\norm{b^\perp ab^\perp}=\inf_{a\in F'}\norm{b^\perp ab^\perp}=\norm{b^\perp pb^\perp}=\mathbf{d}(p,b)^2.\]
Thus, together with the above we have
\begin{equation}\label{supinfFG}
\mathbf{d}(p,q)=\sup_{b\in G}\mathbf{d}(p,b)=\sup_{b\in G}\inf_{a\in F}\mathbf{d}(a,b).
\end{equation}

Now note that
\begin{equation}\label{GsubF}
\sup_{b\in G}\inf_{a\in F}\mathbf{d}(a,b)=0\qquad\Leftrightarrow\qquad G\subseteq F.
\end{equation}
Indeed, the $\mathbf{d}$-initiality of $F$ yields $\Leftarrow$, while the fact $F$ is $\mathbf{d}$-closed yields $\Rightarrow$.  Combined with \eqref{supinfFG}, this shows that $p=q$ implies $F=G$, i.e., the map $F\mapsto\bigwedge F$ is injective on $\mathbf{d}$-filters.  Thus the given maps are bijections, as required.
\end{proof}

In the above proof, we used dense $\leq$-filter subsets of $\mathbf{d}$-filters in a couple of places, but this was not absolutely necessary.  Indeed, one could verify directly that pointwise infimums on $\mathsf{Q}$ of $\mathbf{h}$-directed subsets are affine and hence define elements of $A^{**}$.  Likewise, Dini's theorem can be generalized to $\mathbf{h}$-directed subsets and even $\mathbf{h}$-Cauchy nets \textendash\, see \cite[Theorem 1]{Bice2016c}.

The gist of \autoref{CompactFilter} is that compact projections in $A^{**}$ can be more concretely represented by $\mathbf{d}$-filters in $A$, and this extends to various relations or functions one might consider.  For example, from \eqref{supinfd} and \eqref{GsubF} we immediately see that, for compact projections $p,q\in A^{**}$ and corresponding $\mathbf{d}$-filters $F,G\subseteq A^1_+$,
\[p\leq q\qquad\Leftrightarrow\qquad F\supseteq G.\]
Likewise, as $\norm{p-q}=\max\{\norm{pq^\perp},\norm{qp^\perp}\}$, \eqref{supinfd} yields
\[\norm{p-q}=\max(\sup_{b\in G}\inf_{a\in F}\mathbf{d}(a,b),\sup_{a\in F}\inf_{b\in G}\mathbf{d}(b,a)),\]
i.e., the metric on compact projections corresponds to the Hausdorff metric on $\mathbf{d}$-filters.  We can also show that the natural quantification of orthogonality on compact projections is determined by the corresponding $\mathbf{d}$-filters.

\begin{thm}
For compact $p,q\in A^{**}$ and $\mathbf{d}$-filters $F=\ \uparrow\!p$ and $G=\ \uparrow\!q$,
\[\norm{pq}=\inf_{a\in F,b\in G}\norm{ab}.\]
\end{thm}

\begin{proof}
Let $r=\inf_{a\in F,b\in G}\norm{ab}$.  As $p\ll F$ and $q\ll G$, we immediately have
\[\norm{pq}^2=\norm{qpq}\leq\inf_{a\in F}\norm{qa^2q}=\inf_{a\in F}\norm{aqa}\leq\inf_{a\in F,b\in G}\norm{ab^2a}=r^2.\]
Conversely, take a dense $\leq$-filter $F'\subseteq F$ and, for any $a\in F'$, consider
\[\mathsf{Q}_a=\{\phi\in\mathsf{Q}:\phi[G]=\{1\}\text{ and }\phi(a)\geq r^2\}.\]
By \cite[Theorem 2.2]{Bice2011}, each $\mathsf{Q}_a$ is non-empty.  So $\bigcap_{a\in F'}\mathsf{Q}_a$ is a directed intersection of non-empty weak* compact subsets and we thus have some $\phi\in\bigcap_{a\in F}\mathsf{Q}_a$.  As $\phi[G]=\{1\}$, $\phi(q)=1$ and hence $\norm{pq}^2=\norm{qpq}\geq\phi(qpq)=\phi(p)\geq r^2$.
\end{proof}

A natural question to ask is if the infimum above is actually a minimum.

\begin{qst}[\cite{MO2015}]
Do we always have $a\in F$ and $b\in G$ with $\norm{pq}=\norm{ab}$?
\end{qst}

When $pq=0$ the answer is yes, by Akemann's non-commutative Urysohn lemma \textendash\, see \cite[Lemma III.1]{Akemann1971}.  However, we feel that a truly non-commutative Urysohn lemma should apply to compact projections that do not commute.

Dual to compact projections, we have open projections.  Specifically, let $\downarrow\!\!p$ denote the lower set in $A^1_+$ defined by any projection $p\in A^{**}$, i.e.
\[\downarrow\!p=\{a\in A^1_+:a\leq p\}.\]

\begin{dfn}
A projection $p\in A^{**}$ is \emph{open} if $p=\bigvee\downarrow\!p$ and \emph{closed} if $p^\perp$ is \nolinebreak open.
\end{dfn}

For $\mathbf{D},\mathbf{E}\colon X\times X\to[0,\infty]$, define
\[\widetilde{\mathbf{D}}\circ\mathbf{E}=\sup_{r\in\mathbb{R}}((r\mathbf{D})\circ\mathbf{E}),\]
So $(\widetilde{\mathbf{D}}\circ\mathbf{E})(x,y)=\inf\{r:\forall\epsilon>0\ \exists z\in X\ (x\mathbf{D}z<\epsilon\text{ and }z\mathbf{E}y<r)\}$.  In particular,
\begin{gather*}
\mathbf{D}\circ\mathbf{E}\ \leq\ \widetilde{\mathbf{D}}\circ\mathbf{E}\ \leq\ |\mathbf{D}|\circ\mathbf{E}.\\
\mathbf{C}\precapprox\mathbf{D}\quad\Rightarrow\quad\widetilde{\mathbf{C}}\circ\mathbf{E}\leq\widetilde{\mathbf{D}}\circ\mathbf{E}.
\end{gather*}

The following proof was inspired by interpolation arguments introduced by Brown \textendash\, see \cite{Brown1988} \textendash\, and adapted by Akemann and Pedersen \textendash\, see \cite{AkemannPedersen1992} (although the distance-like functions they used were never formalized as such).

\begin{thm}\label{compactequivalents}
Assume $p\in A^{**}$ is a closed projection.  On $A^1_+$, we have
\begin{equation}\label{plle}
p(\ll\circ\ \mathbf{e})\precapprox p\mathbf{d}
\end{equation}
\end{thm}

\begin{proof}\
If $\inf_{a\in A^1_+}\mathbf{d}(p,a)>0$ then \eqref{plle} holds vacuously on $A^1_+$.  So assume $\inf_{a\in A^1_+}\mathbf{d}(p,a)=0$.  We first claim a weakened form of \eqref{plle} on $A^1_+$, namely
\begin{equation*}
p(\widetilde{\mathbf{d}}\circ\mathbf{e})\precapprox p\mathbf{d}.
\end{equation*}
In more classical terms, we are claiming that
\begin{gather}\label{weakplle}
\forall\epsilon>0\ \exists\delta>0\ \forall a\in A^1_+\\
\nonumber\mathbf{d}(p,a)<\delta\quad\Rightarrow\quad\forall\gamma>0\ \exists b\in A^1_+\ (\mathbf{d}(p,b)<\gamma\text{ and }\norm{a-b}<\epsilon).
\end{gather}
To see this, take $\epsilon>0$.  By \autoref{h=p}, we can take $\delta>0$ such that
\begin{equation}\label{h<4delta}
\mathbf{h}(a,b)<4\delta\quad\Rightarrow\quad\mathbf{p}(a,b)<\tfrac{1}{2}\epsilon,
\end{equation}
for all $a,b\in A^1_+$.  Now take $a\in A^1_+$ with $\mathbf{d}(p,a)<\delta$.  As $\inf_{a\in A^1_+}\mathbf{d}(p,a)=0$ and $A^1_+$ is $\mathbf{d}$-directed, for any $\gamma>0$, we have $u\in A^1_+$ with $\mathbf{d}(a,u),\mathbf{d}(p,u)<\gamma^2$.  By \autoref{h=n}, $\mathbf{n}\precapprox\mathbf{h}\leq2\mathbf{d}$ so we may also assume $a\leq u$ (alternatively, use \cite[Proposition 1]{Brown2015}) and hence $u-a\in A^1_+$.  We may further assume $\gamma^2<\delta$ so
\begin{align*}
\mathbf{h}(u-a,p^\perp)&\leq2\mathbf{d}(u-a,p^\perp)\\
&=2\norm{(u-a)p}\\
&\leq2(\norm{(u-ua)p}+\norm{(ua-a)p})\\
&\leq2(\norm{p-ap}+\norm{ua-a})\\
&=2(\mathbf{d}(p,a)+\mathbf{d}(a,u))\\
&<4\delta.
\end{align*}
As $u-a-p^\perp$ is the pointwise infimum on $\mathsf{Q}$ of $(u-a-c)_{p^\perp\geq c\in A^1_+}$, Dini's theorem again yields $c\in A^1_+$ with $c\leq p^\perp$ and $\mathbf{h}(u-a,c)<4\delta$.  By \eqref{h<4delta}, $\mathbf{p}(u-a,c)<\frac{1}{2}\epsilon$, i.e., we have $d\in A^1_+$ with $\norm{u-a-d}<\frac{1}{2}\epsilon$ and $d\leq c$.  Setting $b=(u-d)_+$,
\[\norm{u-d-b}=\norm{(d-u)_+}\leq\norm{(d+a-u)_+}\leq\norm{d+a-u}<\tfrac{1}{2}\epsilon\]
so $\norm{a-b}\leq\norm{a+d-u}+\norm{u-d-b}<\epsilon$.  As $pd=0$ and $u-d\leq b$,
\[\mathbf{d}(p,b)^2\leq\mathbf{d}(p,u-d)+\mathbf{h}(u-d,b)=\mathbf{d}(p,u)<\gamma^2,\]
thus proving \eqref{weakplle}.

Now \eqref{plle} is saying the same thing as \eqref{weakplle}, just with $\mathbf{d}(p,b)<\gamma$ strengthened to $p\ll b$.  To prove this, we iterate \eqref{weakplle}.  First take $\delta_n>0$ satisfying \eqref{weakplle} with $\epsilon$ replaced by $\epsilon/2^n$, for any fixed $\epsilon>0$.  So for any $a_1\in A^1_+$ with $\mathbf{d}(p,a_1)<\delta_1$, we can recursively take $a_{n+1}\in A^1_+$ with $\mathbf{d}(p,a_{n+1})<\delta_{n+1}$ and $\norm{a_n-a_{n+1}}<\epsilon/2^n$.  Thus $(a_n)$ has a limit $b\in A^1_+$ with $\mathbf{d}(p,b)\leq\mathbf{d}(p,a_n)+\mathbf{e}(a_n,b)\rightarrow0$, i.e., $p\ll b$.  Also $\norm{a_1-b}<\sum\epsilon/2^n=\epsilon$, thus proving \eqref{plle}.
\end{proof}

As one might expect, in the commutative case an easier proof of a stronger result is available.  Specifically, if $a\in A^1_+$ and $p$ is a projection with $ap=pa$ and $\mathbf{d}(p,a)=r<1$ then $p\ll f(a)\in A^1_+$, for any continuous function $f$ on $[0,1]$ taking $0$ to $0$ and $[1-r,1]$ to $1$.  Thus $\inf_{p\ll b\in A^1_+}\norm{a-b}=r$, i.e.
\[(\ll\circ\ \mathbf{e})(p,a)=\mathbf{d}(p,a).\]

In particular, for any projection $p$, $\inf_{a\in A^1_+,ap=pa}\mathbf{d}(p,a)$ must be $0$ or $1$.  However, this is not true for $\inf_{a\in A^1_+}\mathbf{d}(p,a)$.  For example, let $A$ be the C*-subalgebra of $C([0,1],M_2)$ with $f(0)\in\mathbb{C}Q$ for some fixed rank one $Q\in M_2$.  Take any other rank one $P\in M_2$ and define $p$ on $[0,1]$ by $p(0)=Q$ and $p(x)=P$ otherwise.  This represents a closed projection in $A^{**}$ (as the atomic representation is faithful on closed projections \textendash\, see \cite[Theorem 4.3.15]{Pedersen1979}) with $\inf_{a\in A^1_+}\mathbf{d}(p,a)=\norm{P-Q}$, which can be anywhere between $0$ and $1$.

Now we can show that `compact' is the same as `closed and bounded'.  Indeed, this is usually taken as the definition, i.e., compact projections are usually defined as closed projections satisfying some notion of boundedness, like $p\leq a\in A_+$ (see \cite[Definition II.1]{Akemann1971}) or $p\ll a\in A^1_+$ (see \cite[\S3.5]{OrtegaRordamThiel2012}).  We also mention some other boundedness notions below in \autoref{BoundedEqvs}.  However, we feel this obscures the duality between compact and open projections and it is more natural to define them independently via $\uparrow\!p$ and $\downarrow\!p$ respectively, as done here.

\begin{cor}
A projection $p\in A^{**}$ is compact if and only if $p$ is closed and
\[\inf_{a\in A^1_+}\mathbf{d}(p,a)=0.\]
\end{cor}

\begin{proof}
If $p$ is compact then $\uparrow\!p$ is non-empty so certainly $\inf_{a\in A^1_+}\mathbf{d}(p,a)=0$.  To see that $p$ is closed, consider
\[B=\{a\in A:ap=0=pa\},\]
which is immediately seen to be a (hereditary) C*-(sub)algebra.  So $B^1_+$ is $\mathbf{d}$-directed and hence has a supremum $\bigvee B^1_+$ in $A^{**}$ which is a projection and also a pointwise supremum on $\mathsf{Q}$.  For any $\phi\in\mathsf{Q}$ with $\phi(p)=0$, we have $a_n\in\ \uparrow\!p$ with $\phi(a_n)\rightarrow0$.  We also have $b_n\in A$ with $\phi(b_n)\rightarrow\norm{\phi}$ so $a_n^\perp ba_n^\perp\in B^1_+$ and $\phi(a_n^\perp ba_n^\perp)\rightarrow\norm{\phi}$ (by the GNS construction).  Thus $p^\perp=\bigvee B^1_+$ is open so $p$ is closed.

If $p$ is closed then, by \eqref{plle},
\begin{equation}\label{inf=>ll}
\inf_{a\in A^1_+}\mathbf{d}(p,a)=0\qquad\Rightarrow\qquad\exists a\in A^1_+\ (p\ll a).
\end{equation}
Alternatively, note $\inf_{a\in A^1_+}\mathbf{d}(p,a)=0$ implies the facial support $\{\phi\in\mathsf{Q}:\phi(p)=1\}$ is weak* closed in $\mathsf{Q}$ so \cite[Lemma 2.4]{AkemannAndersonPedersen1989} yields \eqref{inf=>ll}.  In any case, we can take $a\in\ \uparrow\!p$.  For all $b\in\ \downarrow p^\perp$, we then have $ab^\perp a\in\ \uparrow\!p$.  Also, as $p$ is closed, i.e., $p^\perp$ is open, we have $p=p^{\perp\perp}=(\bigvee\downarrow\!p^\perp)^\perp=\bigwedge(\downarrow\!p^\perp)^\perp$ so
\[p=apa=\bigwedge a(\downarrow\!p^\perp)^\perp a=\bigwedge\{ab^\perp a:b\in\ \downarrow\!p^\perp\}\geq\bigwedge\uparrow\!p\geq p.\]
(for the second equality note, as $\phi(a\cdot a)\in\mathsf{Q}$ whenever $a\in A^1_+$ and $\phi\in\mathsf{Q}$, $\inf_{c\in C}\phi(c)=\phi(d)$, for all $\phi\in\mathsf{Q}$, implies $\inf_{c\in C}\phi(aca)=\phi(ada)$, for all $\phi\in\mathsf{Q}$).  Thus $p$ is compact.
\end{proof}

For \eqref{inf=>ll}, it is crucial for $p$ to be closed.

\begin{thm}\label{boundedopen}
It is possible to have open $p\in A^{**}$ with
\[\inf_{a\in A^1_+}\mathbf{d}(p,a)=0\qquad\text{but}\qquad\nexists a\in A^1_+\ (p\ll a).\]
\end{thm}

\begin{proof}
We consider a variant of the non-regular open dense projection considered in \cite[Example 4]{AkemannBice2014}, where $A$ is a C*-subalgebra of $C([0,1],\mathcal{B}(H))$, i.e., the continuous functions from $[0,1]$ to $\mathcal{B}(H)$ for a separable infinite dimensional Hilbert space $H$.  First let $(x_n)$ be a countable dense subset of $(0,1)$ (actually, it suffices to have $\inf_nx_n=0$), let $(e_n)$ be an orthonormal basis for $H$ and let $(P_n)$ be the rank $1$ projections onto $(\mathbb{C}e_n)$.  Define $p_n:[0,1]\rightarrow\mathcal{B}(H)$ by
\[p_n(x)=\begin{cases}P_n&\text{if }x>r_n\\ 0&\text{if }x\leq r_n.\end{cases}\]
Let $Q$ be the projection onto $\mathbb{C}v$, for $v=\sum2^{-n}e_n$, and define $q:[0,1]\rightarrow\mathcal{B}(H)$ by
\[q(x)=\begin{cases}Q&\text{if }x>0\\ 0&\text{if }x=0.\end{cases}\]
Let $B$ and $C$ be the hereditary C*-subalgebras of $K=C([0,1],\mathcal{K}(H))$ defined by $p=\bigvee p_n$ and $p'=p\vee q$, i.e.
\[B=pKp\cap K\qquad\text{and}\qquad C=p'Kp'\cap K.\]
Let $A$ be the C*-subalgebra of $C([0,1],\mathcal{B}(H))$ generated by $C$ and the constant projection $Q^\perp$.  Let $a_n=Q^\perp+f_nq\in A$, for some continuous function $f_n$ on $[0,1]$ with $f_n(0)=0$ and $f_n(x)=1$, for all $x\in[\frac{1}{n},1]$.  Then
\[\sup_{b\in B}\mathbf{d}(b,a_n)=\mathbf{d}(p,a_n)=\norm{p(1-f_n)Q}\rightarrow0,\]
i.e., $\inf_{a\in A^1_+}\mathbf{d}(p_B,a)=0$ for $p_B=\bigvee B^1_+\in A^{**}$.  However, for each $x\in(0,1)$,
\[\{a(x):a\in A\}=\mathbb{C}1+\mathcal{K}(q'Hq').\]
Thus if $p\leq a\in A^1_+$ then, for all $x\in(0,1)$, we must have $a(x)=1-f(x)q'(x)$, where $q'(x)=(p\vee q-p)(x)$ and $f$ is some continuous function on $[0,1]$.  But $q'$ is discontinuous at each $r_n$, so the only way $a$ could be continuous is if $f(r_n)=0$ so $a(r_n)=1$ and hence $qa(r_n)=Q$, for all $n$.  But then continuity yields $qa(0)=Q$, contradicting the fact $qa(0)=0$, by the definition of $A$.  Thus there is no $a\in A^1_+$ with $p_B\ll a$.
\end{proof}

There are several other boundedness conditions on $p$ that one might consider.  However, they are all equivalent, even in a more general context.

\begin{prp}\label{BoundedEqvs}
For any $a\in A^1_+$, $r>1$ and C*-subalgebra $B\subseteq A$, TFAE.
\begin{enumerate}
\item\label{a<=b} $\exists b\in B^r_+\ (a\leq b)$.
\item\label{hab} $\inf_{b\in B_\mathrm{sa}}\mathbf{h}(a,b)=0$.
\item\label{dab+} $\inf_{b\in B^1_+}\mathbf{d}(a,b)=0$.
\item\label{dab} $\inf_{b\in B}\mathbf{d}(a,b)=0$.
\end{enumerate}
\end{prp}

\begin{proof}  We immediately have \eqref{a<=b}$\Rightarrow$\eqref{hab} and \eqref{dab+}$\Rightarrow$\eqref{dab}.
\begin{itemize}
\item[\eqref{hab}$\Rightarrow$\eqref{dab+}]  By \eqref{d2<dh} (and the existence of an approximate unit for $B$ in $B^1_+$),
\[\inf_{b\in B^1_+}\mathbf{d}(a,b)^2\leq\inf_{c\in B_\mathrm{sa}}\inf_{b\in B^1_+}(\mathbf{h}(a,c)+\mathbf{d}(c,b))=\inf_{c\in B_\mathrm{sa}}\mathbf{h}(a,c).\]

\item[\eqref{dab}$\Rightarrow$\eqref{hab}]  If $\mathbf{d}(a,b_n)\rightarrow 0$, i.e., $ab_n\rightarrow a$ and hence $b_n^*ab_n\rightarrow a$, then
\[\mathbf{h}(a,b_n^*b_n)\leq\mathbf{h}(a,b_n^*ab_n)\leq\mathbf{e}(a,b_n^*ab_n)\rightarrow0.\]

\item[\eqref{dab+}$\Rightarrow$\eqref{a<=b}]  See \cite[Theorem 1.2]{Akemann1970b}.\qedhere
\end{itemize}
\end{proof}

Another relation on compact relations one might like to quantify is `interior containment'.  Specifically, define the \emph{interior} $p^\circ$ of any projection $p\in A^{**}$ to be the largest open projection below $p$, i.e.
\[p^\circ=\bigvee\downarrow\!p.\]
We quantify the interior containment relation $p\leq q^\circ$ by the distance
\[\mathbf{c}(p,q)=\norm{p-pq^\circ}.\]
By Akemann's non-commutative Urysohn lemma \textendash\, see \cite[Lemma III.1]{Akemann1971} \textendash\, $p\leq q^\circ$ is equivalent to $\exists a\in A^1_+\ (p\leq a\leq q)$.  For commutative $A$, this means
\[\mathbf{c}(p,q)=\inf_{a\in F}\sup_{b\in G}\mathbf{d}(a,b),\]
where $F=\ \uparrow\!p$ and $G=\ \uparrow\!q$.  However, this does not extend non-commutative $A$ and in general $\mathbf{c}$ can behave quite badly with respect to the metric $\mathbf{e}$.

\begin{thm}
It is possible to have compact $p,q\in A^{**}$ with
\begin{equation}\label{infsup}
\inf_{a\in F}\sup_{b\in G}\mathbf{d}(a,b)=0\qquad\text{but}\qquad p\nleq q^\circ,
\end{equation}
where $F=\ \uparrow\!p$ and $G=\ \uparrow\!q$.  It is also possible that $\mathbf{c}\not\precapprox\mathbf{e}\circ\mathbf{c}$ on compact projections.
\end{thm}

\begin{proof}
Take open $p\in A^{**}$ as in \autoref{boundedopen}.  Consider $p=\bigvee\downarrow\!p$ in $\widetilde{A}^{**}$ and let $q=\bigvee A^1_+\in\widetilde{A}$ so $0=\inf_{a\in\downarrow\!q}\mathbf{d}(p,a)=\inf_{a\in\downarrow\!q}\sup_{b\in\downarrow p}\mathbf{d}(p,a)$ but $\nexists a\in\ \downarrow\!q\ (p\ll a)$.  In unital C*-algebras, compact $\Leftrightarrow$ closed and $\mathbf{d}(a,b)=\mathbf{d}(b^\perp,a^\perp)$ so
\[0=\inf_{a\in\uparrow q^\perp}\sup_{b\in\downarrow p^\perp}\mathbf{d}(a,p^\perp),\]
even though there no $a\in\ \uparrow\!q^\perp$ with $a\ll p^\perp$, i.e., $q^\perp\nleq p^{\perp\circ}$.

For $\mathbf{c}\not\precapprox\mathbf{c}\circ\mathbf{e}$, let $A=([0,1],M_2)$.  Let $P_\theta$ be the projection onto $\mathbb{C}(\sin\theta,\cos\theta)$,
\[P_\theta=\begin{bmatrix}\sin^2\theta&\sin\theta\cos\theta\\ \sin\theta\cos\theta&\cos^2\theta\end{bmatrix}.\]
For $\epsilon\geq0$, consider the compact projections $p_\epsilon$ represented by
\[p_\epsilon(x)=\begin{cases}P_{\epsilon\sin(1/x)}&\text{if }x>0\\ 1&\text{if }x=0\end{cases}\]
(this is a projection in the atomic representation of $A$ rather than the universal representation $A^{**}$ but again this does not matter as the atomic representation is faithful on open and closed projections, by \cite[Theorem 4.3.15]{Pedersen1979}).  So $p(x)$ is a rank 1 projection which `wiggles' with amplitude $\epsilon$ and increasing frequency as $x\rightarrow0$.  This means that, for $\epsilon>0$, any $a\in A$ with $a\leq p_\epsilon$ must satisfy $a(0)=0$ so
\[p_\epsilon^\circ(x)=\begin{cases}P_{r\sin(1/\theta)}&\text{if }x>0\\ 0&\text{if }x=0.\end{cases}\]
Also let $p$ be the compact projection defined by $p(x)=P_0$, for all $x\in[0,1]$, so $p_0^\circ=p$.  For all $\epsilon>0$, $\mathbf{c}(p,p_\epsilon)\geq\norm{p(0)-p(0)p_\epsilon^\circ(0)}=\norm{P_0}=1$ even though
\[(\mathbf{c}\circ\mathbf{e})(p,p_\epsilon)\leq\mathbf{c}(p,p_0)+\mathbf{e}(p_0,p_\epsilon)=\mathbf{e}(p_0,p_\epsilon)=\norm{P_0-P_\epsilon}\rightarrow0,\quad\text{as }\epsilon\rightarrow0.\qedhere\]
\end{proof}

In other words, $\mathbf{c}$ fails to be $\mathbf{e}$-invariant in a strong way.  This suggests that the `reverse Hausdorff distance' $\inf_{a\in F}\sup_{b\in G}\mathbf{d}(a,b)$ may actually be the more natural extension of interior containment to non-commutative $A$.  This is especially so if one wants to consider $\mathbf{d}$-filters in a domain theoretic way \textendash\, see \cite{BDD}.

We finish by showing that this distance can also be calculated from $\mathbf{h}$.

\begin{thm}
For any $\mathbf{d}$-filters $F,G\subseteq A^1_+$,
\[\inf_{a\in F}\sup_{b\in G}\mathbf{d}(a,b)=\inf_{a\in F}\sup_{b\in G}\mathbf{h}(a,b).\]
\end{thm}

\begin{proof}
As in the proof of \autoref{CompactFilter}, for $q=\bigwedge G\in A^{**}$ we have
\[\inf_{a\in F}\sup_{b\in G}\mathbf{d}(a,b)=\inf_{a\in F}\mathbf{d}(a,q)\quad\text{and}\quad\inf_{a\in F}\sup_{b\in G}\mathbf{h}(a,b)=\inf_{a\in F}\mathbf{h}(a,q).\]
Let $a_S\in A^{**}$ denote the spectral projection of $a\in A^1_+$ corresponding to $S\subseteq[0,1]$ and consider
\[P=\{a_{[1-\epsilon,1]}:a\in F\text{ and }\epsilon>0\}.\]
Note $P$ and $F$ are coinitial in each other, with respect to both $\mathbf{d}$ and $\mathbf{h}$, i.e.
\[0=\sup_{a\in F}\inf_{p\in P}\mathbf{d}(p,a)=\sup_{p\in P}\inf_{a\in F}\mathbf{d}(p,a)=\sup_{a\in F}\inf_{p\in P}\mathbf{h}(p,a)=\sup_{p\in P}\inf_{a\in F}\mathbf{h}(a,p).\]
Thus
\begin{align*}
\inf_{a\in F}\mathbf{d}(a,q)&\leq\inf_{a\in F,p\in P}(\mathbf{d}(a,p)+\mathbf{d}(p,q))=\inf_{p\in P}\mathbf{d}(p,q)\\
&\leq\inf_{p\in P,a\in F}\mathbf{d}(p,a)+\mathbf{d}(a,q)=\inf_{a\in F}\mathbf{d}(a,q).
\end{align*}
Likewise $\inf\limits_{a\in F}\mathbf{h}(a,q)=\inf\limits_{p\in P}\mathbf{h}(p,q)$.  Now simply note that, by \autoref{dsadp},
\[\inf_{p\in P}\mathbf{d}(p,q)=\inf_{p\in P}\mathbf{h}(p,q).\qedhere\]
\end{proof}

\bibliographystyle{alphaurl}
\bibliography{maths}

\end{document}